\documentclass[11pt, a4paper]{amsart}
\usepackage{amssymb, array, amsmath, amscd, pdfpages, enumerate, amsthm, fixltx2e, setspace, hyperref, bbm,overpic}
\usepackage[T1]{fontenc}

\DeclareMathOperator*{\Ran}{Ran}
\DeclareMathOperator*{\Ker}{Ker}

\DeclareMathOperator*{\dist}{dist}

\newcommand{\inv}{^{-1}}
\newcommand{\dd}{\mathrm{d}}
\newcommand{\B}{\mathcal{B}}

\newcommand{\F}{\mathcal{F}}
\newcommand{\TT}{\mathbb{T}}
\newcommand{\RR}{\mathbb{R}}
\newcommand{\CC}{\mathbb{C}}

\newcommand{\NN}{\mathbb{N}}
\newcommand{\ZZ}{\mathbb{Z}}
\newcommand{\DD}{\mathbb{D}}
\newcommand{\PP}{\mathbb{P}}

\newcommand{\ep}{\varepsilon}

\newcommand{\mlog}{m_{\mathrm{log}}}
\newcommand{\mmax}{m_{\mathrm{max}}}

\newtheorem{thm}{Theorem}[section]
\newtheorem{prp}[thm]{Proposition}

\theoremstyle{definition}
\newtheorem{rem}[thm]{Remark}
\newtheorem{ex}[thm]{Example}

\numberwithin{equation}{section}

\begin{document}

\title[Optimal rates of decay in the Katznelson-Tzafriri theorem]
{Optimal rates of decay in the Katznelson-Tzafriri theorem for operators on Hilbert spaces}

\author[A.C.S.\ Ng]{Abraham C.S.\ Ng}
\address[A.C.S.\ Ng]{Mathematical Institute, University of Oxford, Andrew Wiles Building, Radcliffe Obs.\ Quarter, Woodstock Road, Oxford, OX2 6GG, UK}
\email{abraham.cs.ng@gmail.com}

\author[D. Seifert]{David Seifert}
\address[D. Seifert]{School of Mathematics, Statistics and Physics, Herschel Building, Newcastle University, Newcastle upon Tyne, NE1 7RU, UK}
\email{david.seifert@ncl.ac.uk}

\begin{abstract}
The Katznelson-Tzafriri theorem is a central result in the asymptotic theory of discrete operator semigroups. It states that for a power-bounded operator $T$ on a Banach space we have $||T^n(I-T)\|\to0$ if and only if $\sigma(T)\cap\TT\subseteq\{1\}$. The main result of the present paper gives a sharp estimate for the \emph{rate} at which this decay occurs for operators on Hilbert space, assuming the growth of the resolvent norms $\|R(e^{i\theta},T)\|$ as $|\theta|\to0$ satisfies a mild regularity condition. This  significantly extends an earlier result by the second author, which covered the important case of polynomial resolvent growth. We further show that, under a natural additional assumption, our condition on the resolvent growth is not only sufficient but also necessary for the conclusion of our main result to hold. By considering a suitable class of Toeplitz operators we show that our theory has natural applications even beyond the setting of normal operators, for which we in addition obtain a more general result.
\end{abstract}


\maketitle

\section{Introduction}\label{sec:intro}

The classical Katznelson-Tzafriri theorem \cite{KT86}  may be stated as follows.

\begin{thm}\label{thm:KT}
Let $X$ be a Banach space and let $T\in\B(X)$ be a power-bounded operator on $X$. Then $\|T^n(I-T)\|\to0$ as $n\to\infty$ if and only if $\sigma(T)\cap\TT\subseteq\{1\}$.
\end{thm}

Here $\TT=\{\lambda\in\CC:|\lambda|=1\}$ denotes the unit circle and a bounded linear operator $T$ is said to be \emph{power-bounded} if $\sup_{n\ge0}\|T^n\|<\infty$. Originally a byproduct of Katznelson and Tzafriri's efforts to simplify earlier proofs of so-called \emph{zero-two laws} for Markov operators \cite{Fog76, OrSu70,Zah86}, Theorem~\ref{thm:KT} is now a cornerstone in the asymptotic theory of  discrete operator semigroups. In order to understand why the result is so useful, suppose that $T$ is a power-bounded linear operator on a reflexive Banach space $X$. By the mean ergodic theorem \cite[Sect.~2.1]{Kre85} the space $X$ splits as the direct sum $X=X_0\oplus X_1$ of the closed invariant subspaces $X_0=\Ker(I-T)$ and $X_1$, the latter being the closure of $\Ran(I-T)$. Moreover, the Ces\`aro averages $\frac1n\sum_{k=0}^{n-1}T^k$, $n\ge1$, of the discrete operator semigroup $(T^n)_{n\ge0}$ converge in the strong operator topology to the projection $P$ onto $X_0$ along $X_1$, that is
$$\bigg\|\frac1n\sum_{k=0}^{n-1}T^kx-Px\bigg\|\to0,\quad n\to\infty,$$
for all $x\in X$. Theorem~\ref{thm:KT} allows us to extend this convergence of the Ces\`aro averages to convergence of the  orbits $(T^nx)_{n\ge0}$  themselves. Indeed, if the assumptions of Theorem~\ref{thm:KT} are satisfied for a general Banach space $X$, then $\|T^nx\|\to0$ as $n\to\infty$ for all $x\in\Ran(I-T)$ and a straightforward approximation argument allows us to extend the conclusion to the closure of $\Ran(I-T)$. Thus if $X$ is once again assumed to be reflexive, and in particular if $X$ is a Hilbert space, then we may combine the Katznelson-Tzafriri theorem with the mean ergodic theorem to deduce that $\|T^nx-Px\|\to0$ for all $x\in X$. 

Observe that if $y$ is an element of the dense subspace $Y=\Ker(I-T)\oplus\Ran(I-T)$ of $X$, then there exists $x\in X$ such that $y=Py+(I-T)x$, and hence $T^ny-Py=T^n(I-T)x,$ $n\ge0$. In particular, having at one's disposal quantified versions of Theorem~\ref{thm:KT} makes it possible to deduce \emph{rates of convergence} of orbits originating at points $y\in Y$ towards their ergodic projection $Py$. We remark  that the case $\sigma(T)\cap\TT=\emptyset$ may be handled at once. Indeed, in this case the spectral radius $r(T)$ of $T$ satisfies $r(T)<1$ and we have $\|T^n\|=O(\alpha^n)$ as $n\to\infty$ for any $\alpha\in(r(T),1)$. In what follows we restrict our attention to the much more interesting case where $\sigma(T)\cap\TT=\{1\}$. As was shown in \cite{Sei15, Sei16}, the rate of convergence in Theorem~\ref{thm:KT} is determined by the rate at which the norms of the resolvent operators $R(e^{i\theta},T)=(e^{i\theta}-T)\inv$ blow up as $|\theta|\to0$. 

Let us consider the important special case of power-like or, as it has come to be called in the field, polynomial resolvent growth, which is to say that $\|R(e^{i\theta},T)\|\le C|\theta|^{-\alpha}$, $0<|\theta|\le\pi$, for some $C>0$ and $\alpha\ge1$. It was shown in \cite[Cor.~3.1]{Sei16} that 
\begin{equation}\label{eq:poly_ub}
\|T^n(I-T)\|=O\bigg(\frac{\log(n)^{1/\alpha}}{n^{1/\alpha}}\bigg),\quad n\to\infty.
\end{equation}
Moreover, if the resolvent bound is sharp then, under mild additional assumptions, $\|T^n(I-T)\|$ can decay no faster than $n^{-1/\alpha}$ as $n\to\infty$. These results hold on general Banach spaces. If $\alpha=1$ then $T$ is a so-called \emph{Ritt operator} and we in fact have the sharper upper bound $\|T^n(I-T)\|=O(n\inv)$ as $n\to\infty$; see \cite{Kom68} and also \cite[Th.~4.5.4]{Nev93} and \cite{Lyu99, NaZe99}. However, it was shown in \cite[Th.~3.6]{Sei16} that for $\alpha>2$ the upper bound in \eqref{eq:poly_ub} cannot in general be improved. On the other hand,  \cite[Th.~3.10]{Sei16} yields the sharp upper bound $\|T^n(I-T)\|=O(n^{-1/\alpha})$, $n\to\infty,$ for all $\alpha\ge1$ if $X$ is a Hilbert space. Earlier theoretical contributions to the topic of polynomial rates of decay in Theorem~\ref{thm:KT} may be found in the works of Dungey \cite{Dun08a, Dun11} and Nevanlinna \cite{Nev93, Nev97}. Both authors were at least partially motivated by concrete applications, which in Nevanlinna's case come from the theory of iterative methods, in Dungey's from the theory of random walks and Markov processes that had inspired the original discovery of Theorem~\ref{thm:KT}; see \cite{Dun07a,Dun07b, Dun08b} and also  \cite{CohLin16, CouSaC90}. Other applications arise for instance in the theory of alternating projections \cite{BadSei16, BadSei17}, the study of certain periodic evolution equations \cite{PauSei19} and even in investigations of infinite systems of coupled ordinary differential equations \cite{PauSei18}. Surveys touching on various aspects of the Katznelson-Tzafriri theorem and its quantified versions may be found in \cite{ChiTom07, Lek17}.

In view of the many important applications of the quantified Katznelson-Tzafriri theorem with polynomial growth, many of which arise in the Hilbert space setting, our main objective in the present paper is to present a substantial extension of the sharp polynomial decay estimate discussed above for operators acting on a Hilbert space. In particular, given a power-bounded operator $T$ on a Hilbert space $X$ with $\sigma(T)\cap\TT=\{1\}$ we would like to know for which continuous non-increasing functions $m\colon(0,\pi]\to(0,\infty)$ it is the case that $\|T^n(I-T)\|=O(m\inv(cn))$ as $n\to\infty$ for some $c>0$ if we have a resolvent bound of the form $\|R(e^{i\theta},T)\|\le m(|\theta|)$, $0<|\theta|\le\pi$. Here and in what follows, given a continuous non-increasing function $m\colon(0,\pi]\to(0,\infty)$ such that $m(\ep)\to\infty$ as $\ep\to0+$, we define a right-inverse $m\inv$ of $m$ by
$$m\inv(s)=\inf\{\ep\in(0,\pi]:m(\ep)\le s\},\quad s\ge m(\pi).$$ 
The above discussion of the polynomial case shows that we do obtain the optimal $m\inv$-rate of decay when $m(\ep)=C\ep^{-\alpha}$, $0<\ep\le\pi$, for some $C>0$ and $\alpha\ge1$. The main result of the present paper, Theorem~\ref{thm:pos_inc}, extends this statement to the much larger class of functions $m$ possessing a mild regularity property which we call \emph{reciprocally positive increase}. This class includes all functions of the polynomial type considered above, but it also contains many others, including in particular all relevant \emph{regularly varying} functions. As we show in Section~\ref{sec:Toep} by considering a special class of analytic Toeplitz operators, there are natural examples even of non-normal operators which fit into the framework of our main result and are not of polynomial type. We then show, in Proposition~\ref{prp:pos_inc_nec}, that for a large class of operators, including all normal operators, our condition of reciprocally positive increase is not only sufficient for the best possible $m\inv$-rate of decay to hold in Theorem~\ref{thm:KT} but also necessary. This shows that Theorem~\ref{thm:pos_inc} is optimal in a very natural sense. Finally, we show in Theorem~\ref{thm:quasi_mult} that for the class of so-called \emph{quasi-multiplication operators}, which again contains all normal operators, one obtains a more general result which requires no additional assumptions on the function $m$ but in the special case where $m$ has reciprocally positive increase recovers the optimal $m\inv$-rate of decay of Theorem~\ref{thm:pos_inc}. Our results may be viewed as discrete analogues of those recently obtained in \cite{RSS19} for bounded $C_0$-semigroups on Hilbert spaces, but despite the absence of unbounded operators the case of discrete operator semigroups poses a number of additional technical challenges of its own. 

Wherever possible, we use natural terminology and notation. This includes, in addition to what has already been introduced, various pieces of standard asymptotic notation and also the convention that $x\lesssim y$ for real quantities $x$ and $y$ signifies the existence of an implicit constant $C>0$ such that $x\le Cy$. We denote the open unit disk by $\DD$, and we let $\NN=\{1,2,3,4,\dots\}$ and $\ZZ_+=\NN\cup\{0\}$. The closure of a subset $\Omega$ of $\CC$ will be denoted by $\overline{\Omega}$, its boundary by $\partial\Omega$. We write $\B(X)$ for the algebra of bounded linear operators on a Banach space $X$, and given $T\in\B(X)$ we denote its spectrum by $\sigma(T)$ and its resolvent set by $\rho(T)$. All Banach and Hilbert spaces will implicitly be assumed to be complex.

\section{An optimal upper bound}

Let $T$ be a power-bounded operator acting on some Banach space $X$ and suppose that $\sigma(T)\cap\TT=\{1\}$. Suppose further that we are given a continuous non-increasing function $m\colon(0,\pi]\to(0,\infty)$ such that $\|R(e^{i\theta},T)\|\le m(|\theta|)$ for $0<|\theta|\le\pi$. It is shown in \cite[Th.~2.11]{Sei16} (see also \cite[Th.~2.5]{Sei15}) that for every $c\in(0,1)$ we have 
\begin{equation}\label{eq:mlog}
\|T^n(I-T)\|=O\big(\mlog\inv(cn)\big),\quad n\to\infty,
\end{equation}
 where
$$\mlog(\ep)=m(\ep)\log\left(1+\frac{m(\ep)}{\ep}\right),\quad 0<\ep\le\pi.$$

As was shown in \cite[Cor.~2.6]{Sei16}, this result is in some sense close to being optimal. We now present a variant of \cite[Cor.~2.6]{Sei16} (see also \cite[Rem.~2.7]{Sei16}), which is a discrete analogue of \cite[Prop.~5.4]{CPSST19} and proves that decay strictly faster than $m\inv(cn)$ for all $c>0$ as $n\to\infty$  is impossible provided $\|R(e^{i\theta},T)\|$ does not grow strictly more slowly than $m(|\theta|)$ as $\theta\to0$. Note that the condition on $m$ imposed here is weaker than that in \cite[Cor.~2.6]{Sei16}. Our conclusion is also weaker, but it contains precisely the optimality statement we require. Here and in what follows we say that a function $m\colon(0,\pi]\to(0,\infty)$ has \emph{reciprocally positive increase} if the map $s\mapsto m(s\inv)$, $s\ge\pi\inv$, has positive increase. Thus $m$ has reciprocally positive increase if and only if there exist constants $c\in(0,1]$, $\alpha>0$ and $\ep_0\in(0,\pi]$ such that
\begin{equation}\label{eq:pos_inc}
\frac{m(t\inv\ep)}{m(\ep)}\ge c t^\alpha,\quad t\ge1,\ 0<\ep\le\ep_0.
\end{equation}
The class of functions which have positive increase will play a central role throughout the remainder of this paper. It contains the polynomial-type functions $m(\ep)=C\ep^{-\alpha}$, $0<\ep\le\pi$, for all $C,\alpha>0$, but it also contains regularly varying functions of the form $m(\ep)=C\ep^{-\alpha}|\log(\ep)|^\beta$, $0<\ep\le\pi$, for $C,\alpha>0$, $\beta\in\RR$, and many others besides; see for instance \cite[Sect.~2]{BaChTo16}, \cite{BiGoTe87} and \cite[Sect.~2]{RSS19} for more details.

\begin{prp}\label{prp:lb}
Let $X$ be a Banach space and let $T\in\B(X)$ be a power-bounded operator such that $\sigma(T)\cap\TT=\{1\}$ and 
\begin{equation}\label{eq:cond}\liminf_{\theta\to0+}\max_{\sigma=\pm1}\theta\|R(e^{\sigma i\theta},T)\|>\sup_{n\ge0}\|T^n\|.
\end{equation}
Let $m\colon(0,\pi]\to(0,\infty)$ be a continuous non-increasing function such that $m(\ep)\to\infty$ as $\ep\to0+$ and 
\begin{equation}\label{eq:res_lb}
\limsup_{\theta\to0}\frac{\|R(e^{i\theta},T)\|}{m(|\theta|)}>0.
\end{equation}
Then there exists $c>0$ such that 
\begin{equation}\label{eq:op_lb}\limsup_{n\to\infty}\frac{\|T^n(I-T)\|}{m\inv(cn)}>0,
\end{equation}
and if $m$ has reciprocally positive increase then \eqref{eq:op_lb} holds for all $c>0$.
\end{prp}

\begin{rem}
If $T=I$ and $m(\ep)=\ep\inv$, $0<\ep\le\pi$, then all assumptions of Proposition~\ref{prp:lb} are satisfied except for \eqref{eq:cond}, and the conclusion \eqref{eq:op_lb} does not hold. Hence the  condition in \eqref{eq:cond} cannot in general be omitted.
\end{rem}

\begin{proof}[Proof of Proposition~\ref{prp:lb}]
Define $m_0\colon(0,\pi]\to(0,\infty)$ by 
$$m_0(\ep)=\sup_{\ep\le|\theta|\le\pi}\|R(e^{i\theta},T)\|,\quad 0<\ep\le\pi.$$ 
Then by \cite[Cor.~2.6 and Rem.~2.7]{Sei16} there exist constants $C_0,c_0>0$ such that $\|T^n(I-T)\|\ge c_0 m_0\inv(C_0 n)$ for all sufficiently large $n\ge1$.  By \eqref{eq:res_lb} we may find $\theta_n\in[-\pi,\pi]$, $n\ge1$, and $c_1>0$ such that $|\theta_n|\to0$ as $n\to\infty$ and $\|R(e^{i\theta_n},T)\|>c_1m(|\theta_n|)$ for all $n\ge1$. In particular, we have
$$m_0(|\theta_n|)>c_1m(|\theta_n|)=m_0\big(m_0\inv(c_1m(|\theta_n|))\big), \quad n\ge1,$$
and hence $|\theta_n|< m_0\inv(c_1m(|\theta_n|))$, $n\ge1$. When $n$ is sufficiently large we may find positive integers $k_n$ such that $C_0k_n\le c_1m(|\theta_n|)\le C_0(k_n+1)$. In particular, we may find $n_0\in\NN$ and $b_n\in[C_0,2C_0)$, $n\ge n_0$, such that $c_1m(|\theta_n|))=b_nk_n$ for all $n\ge n_0$. Note in particular that $k_n\to\infty$ as $n\to\infty$, and moreover $m\inv(c_1\inv b_nk_n)\le|\theta_n|$ for all $n\ge n_0$. Let $c=2C_0c_1\inv$. Then
$$m\inv(ck_n)\le m\inv(c_1\inv b_nk_n)\le m_0\inv(c_1m(|\theta_n|))\le m_0\inv(C_0k_n),\quad n\ge n_0,$$
and making $n_0$ larger if necessary we deduce that
$$\|T^{k_n}(I-T)\|\ge c_0 m\inv(ck_n), \quad n\ge n_0,$$
which proves \eqref{eq:op_lb}. The final statement follows from~\cite[Prop.~2.2]{RSS19}.
\end{proof}

Since $\mlog$ grows faster than $m$ there is in general a discrepancy between the upper bound in \eqref{eq:mlog} and the lower threshold implied by \eqref{eq:op_lb}. The question is: when can we improve the upper bound to $\|T^n(I-T)\|=O(m\inv(cn))$, $n\to\infty$, for some $c>0$? If $X$ is finite-dimensional, then either from the mean ergodic theorem or simply by considering $T$ in Jordan normal form  it is straightforward to see that $\|T^n(I-T)\|=O(\alpha^n)$ as $n\to\infty$ for some $\alpha\in[0,1)$. Since 
$$m(\ep)\ge\|R(e^{i\ep},T)\|\ge\frac{1}{\dist(e^{i\ep},\sigma(T))}\ge\frac{1}{|e^{i\ep}-1|}\ge\frac{1}{\ep},\quad 0<\ep\le\pi,$$
we have $ m\inv(cn)\ge c\inv n\inv$ for all $c>0$ and all  $n\ge1$, and hence in particular $\|T^n(I-T)\|=O(m\inv(cn))$, $n\to\infty$, for all $c>0$. So the question is of real interest only in the setting of infinite-dimensional Banach spaces. Let us return to the important special case where $\|R(e^{i\theta},T)\|\le C|\theta|^{-\alpha}$, $0<|\theta|\le\pi$, for some constants $C>0$ and $\alpha\ge1$. In this case \eqref{eq:mlog} implies the decay estimate \eqref{eq:poly_ub}, and Proposition~\ref{prp:lb} shows that if \eqref{eq:cond} holds and  
$$\limsup_{\theta\to0}|\theta|^{\alpha}\|R(e^{i\theta},T)\|>0,$$ 
then 
$$\limsup_{n\to\infty}n^{1/\alpha}\|T^n(I-T)\|>0.$$
Thus the upper bound in \eqref{eq:poly_ub} is  close to being sharp, and it is never suboptimal by more than a logarithmic factor. As discussed in Section~\ref{sec:intro}, the upper bound in \eqref{eq:poly_ub} is optimal in general but it may be improved to $\|T^n(I-T)\|=O(n^{-1/\alpha})$ as $n\to\infty$ if either $\alpha=1$ (so that $T$ is a Ritt operator) or, crucially for our purposes, if $X$ is a Hilbert space. Theorem~\ref{thm:pos_inc} below, the main result of this paper, is a vast generalisation of the latter statement. It yields the optimal $m\inv$-rate of decay for operators on a Hilbert space under the much milder assumption that the resolvent bound $m$ has reciprocally positive increase. In some ways the result may be viewed as a discrete analogue of \cite[Th.~3.2, 3.6 and~3.9]{RSS19}, but there are additional technical obstacles which are specific to our setting. To overcome these we combine the ideas used in the proofs in \cite{RSS19} with various new ideas, some of which go back to \cite{ChiSei16, Sei15}. We point out that, while any resolvent bound $m$ which has reciprocally positive increase is permitted in the theorem, functions of regularly varying type, whose growth near zero is close to being polynomial, are of especial interest in the present context. Functions which blow up much faster than polynomially will typically also have positive increase (provided the blow-up occurs sufficiently regularly), but in this case the difference in the asymptotic behaviour of $m\inv$ and $\mlog\inv$ usually disappears. 

\begin{thm}\label{thm:pos_inc}
Let $X$ be a Hilbert space and suppose that $T\in\B(X)$ is a power-bounded operator such that $\sigma(T)\cap\TT=\{1\}$. Suppose further that $m\colon(0,\pi]\to(0,\infty)$ is a continuous non-increasing function of reciprocally positive increase such that $\|R(e^{i\theta},T)\|\le m(|\theta|)$ for $0<|\theta|\le\pi$. Then
$$\|T^n(I-T)\|=O\big(m\inv(n)\big),\quad n\to\infty.$$
\end{thm}

\begin{proof}
For the time being we fix $n\in\NN$ and $x\in X$. For $\ell\in\ZZ_+$ and $k\in\ZZ$ let us inductively define non-negative integers $s_\ell(k)$ by $s_0(k)=\chi_{\{1,\dotsc,n\}}(k)$ and $s_{\ell+1}(k)=\smash{\sum_{j=1}^k s_{\ell}(j)}$. In particular, $s_\ell(k)=0$ whenever $k\le0$ and $s_\ell(k)\le\smash{\binom{\ell+k-1}{\ell}}$ for all $k\ge1$, with equality for $1\le k\le n$. Let us take $\ell\ge1$ to be a fixed number, the exact value of which will be chosen in due course, and note that
\begin{equation}\label{eq:start}
T^n(I-T)x=\frac{1}{s_{\ell+1}(n)}\sum_{k=1}^nT^{n-k}h_\ell(k),
\end{equation}
where we define $h_\ell\colon\ZZ\to X$ by $h_\ell(k)=0$ for $k\le0$ and $h_\ell(k)=s_\ell(k)T^k(I-T)x$ for $k\ge1$. Let $\psi\colon\RR\to \RR$ be a smooth function such that $\psi(\theta)=0$ for $|\theta|\le1$, $0\le\psi(\theta)\le1$ for $1\le|\theta|\le2$ and $\psi(\theta)=1$ for $|\theta|\ge2$. Furthermore, for $\ep\in(0,1]$ let $\psi_\ep(\theta)=\psi(\ep\inv\theta)$, $-\pi\le\theta\le\pi$, and define the sequences $y_\ep,z_\ep\in\ell^2(\ZZ)$ by
$$y_\ep(k)=\frac{1}{2\pi}\int_{-\pi}^\pi e^{ik\theta}\psi_\ep(\theta)\,\dd\theta\quad \mbox{and}\quad z_\ep(k)=\frac{1}{2\pi}\int_{-\pi}^\pi e^{ik\theta}(1-\psi_\ep(\theta))\,\dd\theta$$
for all $k\in\ZZ$. Since $\psi$ is smooth the sequences $y_\ep, z_\ep$ are in fact rapidly decaying, so we obtain well-defined convolutions $h_\ell*y_\ep,h_\ell*z_\ep\colon\ZZ\to X$. Furthermore, we have the splitting $h_\ell=h_\ell*y_\ep+h_\ell*z_\ep$, which in turn induces a natural splitting in \eqref{eq:start}. Our aim in the remainder of this proof is to estimate in turn the norms of the two terms
\begin{equation}\label{eq:two_terms}
\frac{1}{s_{\ell+1}(n)}\sum_{k=1}^nT^{n-k}(h_\ell*y_\ep)(k)\quad\mbox{and}\quad \frac{1}{s_{\ell+1}(n)}\sum_{k=1}^nT^{n-k}(h_\ell*z_\ep)(k).
\end{equation}

We begin by estimating the first term. By H\"older's inequality we have
\begin{equation}\label{eq:Holder}
\bigg\|\frac{1}{s_{\ell+1}(n)}\sum_{k=1}^nT^{n-k}(h_\ell*y_\ep)(k)\bigg\|\le \frac{Kn^{1/2}}{s_{\ell+1}(n)}\|h_\ell*y_\ep\|_{\ell^2(\ZZ;X)},
\end{equation}
where $K=\sup_{k\ge0}\|T^k\|$. We now estimate the $\ell^2$-norm of $h_\ell*y_\ep$. Let $\F\colon\ell^2(\ZZ;X)\to L^2(-\pi,\pi;X)$ denote the  discrete-time Fourier transform, defined for $h\in\ell^2(\ZZ;X)\cap\ell^1(\ZZ;X)$ by
$$(\F h)(\theta)=\sum_{k\in\ZZ}e^{-ik\theta}h(k),\quad -\pi\le\theta\le\pi.$$
Since $X$ is a Hilbert space, the operator $(2\pi)^{-1/2}\F$ is unitary by Parseval's theorem, and the inverse of $\F$ is given by
$$(\F\inv \eta)(k)=\frac{1}{2\pi}\int_{-\pi}^\pi e^{ik\theta}\eta(\theta)\,\dd\theta,\quad \eta\in L^2(-\pi,\pi;X),\ k\in\ZZ.$$ 
For $r>1$ let us define $h_{\ell,r}\colon\ZZ\to X$ by $h_{\ell,r}(k)=r^{- k}h_\ell(k)$, $k\in\ZZ$, and define $T_r\colon\ZZ\to\B(X)$ by $T_r(k)=0$ for $k<0$ and $T_r(k)=r^{-k}T^k$ for $k\ge0$. A straightforward calculation shows that
$h_{\ell,r}=T_r*h_{\ell-1,r}$, and hence $h_{\ell,r}=T_r^{*\ell}*h_{0,r}.$ Using the fact that 
$$\sum_{k=0}^\infty e^{-ik\theta}T_r(k)=re^{i\theta}R(re^{i\theta},T),\quad -\pi\le\theta\le\pi,$$
we have
$$(\F h_{\ell,r})(\theta)=r^\ell e^{i\ell\theta}R(re^{i\theta},T)^\ell(\F h_{0,r})(\theta)=R_{\ell,r}(\theta)(\F g_{r})(\theta)$$
for $-\pi\le\theta\le\pi$, where $R_{\ell,r}(\theta)=r^\ell e^{i\ell\theta}R(re^{i\theta},T)^\ell(I-T)$ for $-\pi\le\theta\le\pi$ and where $g_r$ is defined by $g_r(k)=0$ when $k<0$ and $g_r(k)=r^{- k}s_0(k)T^kx$ for $k\ge0$. If we moreover let $R_{\ell}(\theta)=e^{i\ell\theta}R(e^{i\theta},T)^\ell(I-T)$ for $-\pi\le\theta\le\pi$ and define $g$ by $g(k)=0$ for $k<0$ and $g(k)=s_0(k)T^kx$ for $k\ge0$, then for each $\theta\in[-\pi,\pi]$ we have $R_{\ell,r}(\theta)\to R_{\ell}(\theta)$ in $\B(X)$ and $(\F g_{r})(\theta)\to(\F g)(\theta)$ as $r\to1+$. Recalling that $\psi_\ep(\theta)=0$ for $|\theta|\le \ep$ it follows using the dominated convergence theorem that, for each $k\in\ZZ$,
$$\begin{aligned}
(h_\ell*y_\ep)(k)&=\lim_{r\to1+}(h_{\ell,r}*y_\ep)(k)=\lim_{r\to1+} \big(\F\inv (\psi_\ep R_{\ell,r}\,\F g_{r})\big)(k)\\
&=\lim_{r\to1+}\frac{1}{2\pi}\int_{\ep\le|\theta|\le\pi}e^{ik\theta}\psi_\ep (\theta)R_{\ell,r}(\theta)(\F g_{r})(\theta)\,\dd\theta\\
&=\frac{1}{2\pi}\int_{\ep\le|\theta|\le\pi}e^{ik\theta}\psi_\ep (\theta)R_{\ell}(\theta)(\F g)(\theta)\,\dd\theta\\
&=\big(\F\inv (\psi_\ep R_{\ell}\,\F g)\big)(k),
\end{aligned}$$
which is to say that $h_\ell*y_\ep=\F\inv (\psi_\ep R_{\ell}\,\F g)$. Hence by Parseval's theorem 
$$\|h_\ell*y_\ep\|_{\ell^2(\ZZ;X)}\le \|\psi_\ep R_{\ell}\|_{L^\infty(-\pi,\pi;\B(X))}\|g\|_{\ell^2(\ZZ;X)}.$$
It is clear that $\|g\|_{\ell^2(\ZZ;X)}\le Kn^{1/2}\|x\|$, so it remains to estimate the $L^\infty$-norm of $\psi_\ep R_{\ell}$. We note first that since $\|\psi_\ep\|_{L^\infty(-\pi,\pi)}\le1$ and
$\|R(e^{i\theta},T)\|\ge|\theta|\inv$ for $0<|\theta|\le\pi,$
we have $\|R_\ell(\theta)\|\le 2|\theta|\|R(e^{i\theta},T)\|^\ell,$ $0<|\theta|\le\pi$, and hence
$$\|\psi_\ep R_{\ell}\|_{L^\infty(-\pi,\pi; \B(X))}\le 2 \sup_{\ep\le |\theta| \le\pi} |\theta| \,m( |\theta|)^\ell.$$
Since $m$ by assumption has reciprocally positive increase there exist $c\in(0,1]$, $\alpha>0$ and $\ep_0\in(0,\pi]$ such that \eqref{eq:pos_inc} holds. We now make the choice $\ell=\lceil\alpha\inv\rceil$. A brief calculation using the fact that $\sup_{\ep_0\le|\theta|\le\pi}\|R(e^{i\theta},T)\|<\infty$ yields
$$\|\psi_\ep R_{\ell}\|_{L^\infty(-\pi,\pi; \B(X))}\lesssim\ep\bigg(\frac{m(\ep)}{c}\bigg)^\ell,$$
and combining this with \eqref{eq:Holder} we find that 
\begin{equation}\label{eq:first_term_est}
\frac{1}{s_{\ell+1}(n)}\sum_{k=1}^nT^{n-k}(h_\ell*y_\ep)(k)\lesssim \ep\|x\|\bigg(\frac{m(\ep)}{cn}\bigg)^\ell,
\end{equation}
where the implicit constants are independent of $n$, $\ep$, and $x$. Here we have used the fact that 
$$s_{\ell+1}(n)=\binom{\ell+n}{\ell+1}\ge\frac{n^{\ell+1}}{(\ell+1)!}.$$

We now turn to the second term in~\eqref{eq:two_terms}. Define $H_\ell\colon\ZZ\to X$ by $H_\ell(k)=\sum_{j=1}^kh_\ell(j)$. In particular, $H_\ell(k)=0$ for $k\le0$ and for $k\ge1$ we have
$$H_\ell(k)=\sum_{j=1}^ks_{\ell-1}(j)T^{j}x-s_\ell(k)T^{k+1}x,$$
so that $\|H_\ell(k)\|\le 2Ks_\ell(k)\|x\|$, $k\ge1$, where $K$ is as above. Thus
$$\begin{aligned}
\|(h_\ell*z_\ep)(k)\|&=\bigg\|\sum_{j=1}^\infty H_\ell(j)\big(z_\ep(k-j)-z_\ep(k-j-1)\big)\bigg\|\\&\le2K\|x\|(s_\ell*\Delta_\ep)(k),\quad k\in\ZZ,
\end{aligned}$$
where $\Delta_\ep(k)=|z_\ep(k)-z_\ep(k-1)|$, $k\in\ZZ$. For $p\in\ZZ_+$ let us define $\Delta_{\ep,p}$ by $\Delta_{\ep,0}=\Delta_\ep$ and 
$$\Delta_{\ep,p+1}(k)=\begin{cases}
\sum_{j=-\infty}^k\Delta_{\ep,p}(j), & k<0,\\
-\sum_{j=k+1}^\infty\Delta_{\ep,p}(j), &k\ge0.
\end{cases}$$
Then
$$\Delta_{\ep,p+1}(k)-\Delta_{\ep,p+1}(k-1)=\begin{cases}
\Delta_{\ep,p}(k), & k\in\ZZ\setminus\{0\},\\
\Delta_{\ep,p}(0)-\langle\Delta_{\ep,p}\rangle, & k=0,
\end{cases}$$
where $\langle\Delta_{\ep,p}\rangle=\sum_{k\in\ZZ}\Delta_{\ep,p}(k)$, $p\ge0$. A standard argument using summation by parts shows that 
$$(s_\ell*\Delta_\ep)(k)=(s_{\ell-1}*\Delta_{\ep,1})(k)+s_\ell(k)\langle\Delta_\ep\rangle,\quad k\in\ZZ,$$ 
and likewise, by induction,
$$(s_\ell*\Delta_\ep)(k)=(s_{0}*\Delta_{\ep,\ell})(k)+\sum_{j=0}^{\ell-1} s_{\ell-j}(k)\langle\Delta_{\ep,j}\rangle,\quad k\in\ZZ.$$
Thus
$$|(s_\ell*\Delta_\ep)(k)|\le \sum_{j=0}^{\ell} s_{\ell-j}(k)\|\Delta_{\ep,j}\|_{\ell^1(\ZZ)},\quad 1\le k\le n,$$
and it follows that 
\begin{equation}\label{eq:second_term_est}
\begin{aligned}
\bigg\|\frac{1}{s_{\ell+1}(n)}\sum_{k=1}^nT^{n-k}(h_\ell*z_\ep)(k)\bigg\|&\lesssim \|x\|\sum_{j=0}^\ell \frac{s_{\ell-j+1}(n)}{s_{\ell+1}(n)}\|\Delta_{\ep,j}\|_{\ell^1(\ZZ)}\\&\lesssim  \|x\|\sum_{j=0}^\ell \frac{\|\Delta_{\ep,j}\|_{\ell^1(\ZZ)}}{n^j},
\end{aligned}
\end{equation}
with implicit constants which are again independent of $n$, $\ep$ and $x$. It remains to bound the $\ell^1$-norms of $\Delta_{\ep,j}$ for $0\le j\le\ell$. To this end we begin by observing that
$$z_\ep(k)-z_\ep(k-1)=\frac{1}{2\pi}\int_{-\pi}^\pi e^{ik\theta}(1-e^{-i\theta})\varphi(\ep\inv\theta)\,\dd\theta,\quad k\in\ZZ,$$
where $\varphi=1-\psi$. In particular, $\varphi(\theta)=0$ for $|\theta|\ge2$, and hence
\begin{equation}\label{eq:zero_est}
\Delta_\ep(k)\le\frac{1}{2\pi}\int_{-2\ep}^{2\ep}|\theta|\,\dd\theta=\frac{2\ep^2}{\pi},\quad k\in\ZZ.
\end{equation}
Furthermore, integrating by parts $j\ge1$ times we obtain, for $k\in\ZZ\setminus\{0\}$,
$$\int_{-\pi}^\pi e^{ik\theta}(1-e^{-i\theta})\varphi(\ep\inv\theta)\,\dd\theta=\frac{(-1)^j}{(ik)^j}\int_{-\pi}^\pi e^{ik\theta}\frac{\dd^j}{\dd\theta^j}\big((1-e^{-i\theta})\varphi(\ep\inv\theta)\big)\,\dd\theta,$$
and we have
$$\begin{aligned}
\frac{\dd^j}{\dd\theta^j}\big((1-e^{-i\theta})\varphi(\ep\inv\theta)\big)=(1&-e^{-i\theta})\frac{\varphi^{(j)}(\ep\inv\theta)}{\ep^j}\\&-\sum_{p=0}^{j-1}\binom{j}{p}(-i)^{j-p}e^{-i\theta}\frac{\varphi^{(p)}(\ep\inv\theta)}{\ep^p}.
\end{aligned}$$
Now straightforward estimates show that for each fixed $j\ge1$ we have  
\begin{equation}\label{eq:j_est}
|\Delta_\ep(k)|\lesssim \frac{\ep^{2}}{\ep^jk^{j}},\quad k\in\ZZ\setminus\{0\},\ \ep\in(0,1].
\end{equation}
We now use  \eqref{eq:j_est} with $j=2$ and combine it with \eqref{eq:zero_est} to bound the $\ell^1$-norm of $\Delta_\ep$ by
$$\|\Delta_\ep\|_{\ell^1(\ZZ)}\lesssim \sum_{|k|\le\ep\inv}\ep^2+\sum_{|k|>\ep\inv}\frac{1}{k^2}\lesssim\ep.$$
This in turn yields $|\Delta_{\ep,1}(k)|\lesssim\|\Delta_{\ep,0}\|_{\ell^1(\ZZ)}\lesssim\ep$, $k\in\ZZ$. Using \eqref{eq:j_est} with $j=3$ we find that, for $k\ge0$, 
$$|\Delta_{\ep,1}(k)|\lesssim\sum_{j=k+1}^\infty|\Delta_{\ep,0}(k)|\lesssim\ep\inv\sum_{j=k+1}^\infty\frac{1}{j^3}\lesssim\frac{\ep\inv}{(k+1)^2},$$
and similarly $|\Delta_{\ep,1}(k)|\lesssim\ep\inv k^{-2}$ for $k<0$. Hence
\begin{equation}\label{eq:Delta_1_est}
\|\Delta_{\ep,1}\|_{\ell^1(\ZZ)}\lesssim \sum_{|k|\le\ep\inv}\ep+\sum_{|k|>\ep\inv}\frac{\ep\inv}{k^2}\lesssim1.
\end{equation}
Likewise, starting with \eqref{eq:j_est} for $j=4$ we may show that $|\Delta_{\ep,2}(k)|\lesssim \ep^{-2}k^{-2}$ for $k\in\ZZ\setminus\{0\}$, and combining this with \eqref{eq:Delta_1_est} yields $\|\Delta_{\ep,2}\|_{\ell^1(\ZZ)}\lesssim\ep\inv$. Iterating this process we see that $\|\Delta_{\ep,j}\|_{\ell^1(\ZZ)}\lesssim\ep^{1-j}$ for $0\le j\le\ell$, and plugging this into \eqref{eq:second_term_est} we obtain
\begin{equation}\label{eq:second_term_final_est}
\bigg\|\frac{1}{s_{\ell+1}(n)}\sum_{k=1}^nT^{n-k}(h_\ell*z_\ep)(k)\bigg\|\lesssim  \ep\|x\|\sum_{j=0}^\ell \frac{1}{\ep^jn^j}.
\end{equation}

Combining \eqref{eq:second_term_final_est} with \eqref{eq:first_term_est} and recalling \eqref{eq:start} we find that
\begin{equation}\label{eq:penultimate_est}
\|T^n(I-T)x\|\lesssim\ep\|x\|\Bigg(\bigg(\frac{m(\ep)}{cn}\bigg)^\ell+\sum_{j=0}^\ell \frac{1}{\ep^jn^j}\Bigg),
\end{equation}
where the implicit constant is independent of $n$, $\ep$ and $x$. We now set $\ep=m\inv(n)$ for $n\ge1$ sufficiently large to ensure that $\ep\in(0,1]$. Then $m(\ep)=n$ and hence $\ep n=\ep m(\ep)\ge1$. Since the map $t\mapsto\sum_{j=0}^\ell t^{-j}$ is decreasing on $(0,\infty)$, it follows from \eqref{eq:penultimate_est} that $\|T^n(I-T)\|\lesssim m^{-1}(n)$ for all sufficiently large $n\ge1$. This completes the proof.
\end{proof}

It is straightforward to construct examples of normal operators, and in particular of multiplication operators on $L^2$-spaces, which fit into the framework of Theorem~\ref{thm:pos_inc} and have essentially any desired rate of resolvent growth. As we show in the next section, the kind of resolvent behaviour to which Theorem~\ref{thm:pos_inc} is tailored arises naturally also in the setting of non-normal operators.

\section{A special class of analytic Toeplitz operators}\label{sec:Toep}

We consider here a special class of analytic Toeplitz operators on $H^2(\TT)$. More specifically, we are interested in Toeplitz operators whose
 symbols are not merely in $H^\infty(\TT)$ but of a rather particular form. They are examples of what Dungey \cite{Dun11} refers to as \emph{subordinated discrete semigroups of operators}, which arise naturally in the study of random walks on groups. A good deal of what follows could be extended without much difficulty to the class of \emph{all} analytic Toeplitz operators, but our special class has some nice additional properties and, as we shall see, is sufficiently rich to generate examples of non-normal operators with interesting resolvent behaviour. 

Let us write $P(\ZZ_+)$ for the set of sequences $a=(a_n)_{n\ge0}\in\ell^1(\ZZ_+)$ such that $a_n\ge0$ for all $n\ge0$ and $\sum_{n=0}^\infty a_n=1$. Furthermore, given $a\in P(\ZZ_+)$, we define the function $\phi_a$ by
$$\phi_a(\lambda)=\sum_{n=0}^\infty a_n\lambda^n,\quad |\lambda|\le1.$$
One may think of the elements of $P(\ZZ_+)$ as probability densities on $\ZZ_+$, and of $\phi_a$ as the probability generating function of a random variable $Y$ taking values in $\ZZ_+$ with $\PP(Y=n)=a_n$ for $n\ge0$. Note that, for each $a\in P(\ZZ_+)$,  the function $\phi_a$ lies in the analytic Wiener algebra and hence in the disk algebra. In particular, if we identify $\phi_a$ with its boundary function, which in this case is simply the restriction to the unit circle, then $\phi_a\in H^\infty(\TT)$. Let us denote by $T_a$ the Toeplitz operator on the Hardy space $H^2(\TT)$ with symbol $\phi_a$, so that $T_af(t)=\phi_a(t)f(t)$ for $f\in H^2(\TT)$, $t\in\TT$. Thus $T_a$ is a multiplication operator of unit norm, but it  is non-normal unless $a=(1,0,0,\dots)$, in which case $T_a=I$. We mention in passing that $T_a$ is unitarily equivalent, for each $a\in P(\ZZ_+)$, to the discrete convolution operator $Q_a$ defined on $\ell^2(\ZZ_+)$ by $Q_ax=a*x$. In particular, since $P(\ZZ_+)$ is invariant under $Q_a$, we may think of $a\in P(\ZZ_+)$ as encoding the transition probabilities of a homogeneous Markov chain on $\ZZ_+$ associated with $T_a$.  

One of the main reasons for working with the operators $T_a$ is that their spectral theory is particularly simple. Indeed, it follows from continuity of the symbol $\phi_a$ and the classical theory of analytic Toeplitz operators \cite[Sect.~2.4]{BoeSil06}, or alternatively from \cite[Th.~2.1]{Dun11}, that $\sigma(T_a)=\phi_a(\overline{\DD})$ for all $a\in P(\ZZ_+)$. Furthermore, it is straightforward to see that $\sigma(T_a)\cap\TT=\phi_a(\TT)\cap\TT$. Following \cite{Dun11} we say that a density $a\in P(\ZZ_+)$ is  \emph{aperiodic} if for every $k\in\ZZ$ the set $\{n\in\ZZ:n\ge k\mbox{ and } a_{n-k}>0\}$ generates the additive group $\ZZ$. Note in particular that $a\in P(\ZZ_+)$ is aperiodic whenever there exists $n\in\ZZ_+$ such that $a_n,a_{n+1}>0$. It is shown in \cite[Prop.~2.5]{Dun11} and \cite[Chap.~II]{Spi76} that $\phi_a(\TT)\cap\TT=\{1\}$ whenever $a\in P(\ZZ_+)$ is an aperiodic density, and that the converse holds for densities $a\in P(\ZZ_+)$ whose support $\{n\in\ZZ_+:a_n>0\}$ generates $\ZZ$. Hence if $a\in P(\ZZ_+)$ is aperiodic and if $0<|\theta|\le\pi$ then $R(e^{i\theta},T_a)$ is the (analytic) Toeplitz operator with symbol $t\mapsto (e^{i\theta}-\phi_a(t))\inv$, $t\in\TT$, and in particular 
\begin{equation}\label{eq:Toep_res}
\|R(e^{i\theta},T_a)\|=\frac{1}{\dist(e^{i\theta},\phi_a(\TT))}, \quad 0<|\theta|\le\pi.
\end{equation}
This allows us to construct many examples of non-normal operators which fit into the framework of Theorem~\ref{thm:pos_inc} and have interesting resolvent behaviour. 

\begin{ex}\label{eq:Toep}
If we take $a\in P(\ZZ_+)$ to be the aperiodic density defined by $a_0=a_1=0$ and $a_n=n\inv(n-1)\inv$, $n\ge2$, then we find as in the proof of \cite[Th.~5.2]{Dun11} that 
$$\phi_a(\lambda)=\lambda+(1-\lambda)\log(1-\lambda),\quad \lambda\in\overline{\DD}\setminus\{1\},$$
where $\log$ denotes the principal branch of the logarithm, with a branch cut along $(-\infty,0]$. A direct calculation using \eqref{eq:Toep_res} shows that 
$$\|R(e^{i\theta},T_a)\|\sim\frac{2\log|\theta|\inv}{\pi|\theta|},\quad|\theta|\to0,$$
and in particular the best possible upper bound for the resolvent in this example will have reciprocally positive increase, indeed will be regularly varying, but will not be of polynomial type. We deduce from Theorem~\ref{thm:pos_inc} after a small calculation that $\|T_a^n(I-T_a)\|=O(n\inv\log (n))$ as $n\to\infty$, and in fact from  \cite[Cor.~2.6]{Sei16} we have $\|T_a^n(I-T_a)\|\asymp n\inv\log(n)$ as $n\to\infty$. Note that the decay rate implied by \eqref{eq:mlog} in this case is $\|T_a^n(I-T_a)\|=O(n\inv\log (n)^2)$ as $n\to\infty$, which is off by a full logarithm.
\end{ex}

\begin{rem}\begin{enumerate}[(a)]
\item It is shown in \cite[Th.~2.4]{Dun11} that $\|T_a^n(I-T_a)\|=O(n^{-1/2})$ as $n\to\infty$ for \emph{all} aperiodic densities $a\in P(\ZZ_+)$, and by considering the sequence $a=(\frac12,\frac12,0,0,\dots)$ and appealing to Proposition~\ref{prp:lb} or \cite[Cor.~2.6]{Sei16} we see that this bound cannot in general be improved. 
\item A more involved approach leading to general resolvent estimates for Toeplitz operators with \emph{continuous} symbols may be found in \cite{Pel84}.
\end{enumerate}
\end{rem}
 
\section{Necessity of reciprocally positive increase}

Having shown in Theorem~\ref{thm:pos_inc} that we obtain the best possible rate of decay in Theorem~\ref{thm:KT} whenever the resolvent bound is a function of reciprocally positive increase, we now show that this is the largest possible class of functions for which one may hope to obtain such a statement, at least within a natural class of operators for which the resolvent norms are determined (up to a constant) by the distance to the spectrum. This latter condition is satisfied in many natural examples, and in particular it holds for all normal operators and, more generally, for multiplication operators on classical function spaces such as $L^p$-spaces and spaces of continuous functions. The condition is also satisfied for the larger class of \emph{quasi-multiplication operators} to be introduced in Section~\ref{sec:QMO} below, which includes non-normal (multiplication) operators on the Hardy space $H^2(\TT)$ of the type considered in Section~\ref{sec:Toep}; see Example~\ref{ex:QMO}. Thus Theorem~\ref{thm:pos_inc} is optimal in a rather natural sense; for related statements in the continuous case see \cite[Th.~3.4, 3.7 and~3.10]{RSS19}.

\begin{prp}\label{prp:pos_inc_nec}
Let $X$ be a Banach space and let $T\in\B(X)$ be such that $\{1\}\subseteq\sigma(T)\cap\TT\subseteq\DD\cup\{1\}$. Further let $m\colon(0,\pi]\to(0,\infty)$ be a continuous non-increasing function and suppose that $\|R(e^{i\theta},T)\|\le m(|\theta|)$ for $0<|\theta|\le\pi$ and that
$$ \max_{\ep\le|\theta|\le\pi}\frac{1}{\dist(e^{i\theta},\sigma(T))}\ge\delta m(\ep),\quad 0<\ep\le\pi,$$
for some $\delta\in(0,1]$. If 
$$\|T^n(I-T)\|=O\big(m\inv(cn)\big),\quad n\to\infty,$$
for some $c>0$ then $m$ has reciprocally positive increase.
\end{prp}

\begin{proof}
Let 
$$p(\ep)=\max_{\ep\le|\theta|\le\pi}\frac{1}{\dist(e^{i\theta},\sigma(T))},\quad 0<\ep\le\pi,$$
so that $\delta m(\ep)\le p(\ep)\le m(\ep)$, $0<\ep\le\pi$, by simple properties of resolvent operators. Using the spectral mapping theorem for polynomials and the fact that $m\inv$ is non-increasing we may find a constant $C>0$ such that for all $\lambda\in\sigma(T)$ and all $b\in(c/2,c]$ we have
$$|\lambda|^n|1-\lambda|\le Cm\inv(bn)$$
for all sufficiently large $n\ge1$, and hence for $\lambda\in\sigma(T)\setminus\{0,1\}$ we have
\begin{equation}\label{eq:est}
n\log\left(\frac{1}{|\lambda|}\right)\ge\log\left(\frac{|1-\lambda|}{Cm\inv(bn)}\right)\ge\log\left(\frac{|1-\lambda|}{Cp\inv(\delta bn)}\right).
\end{equation}
Given $\ep\in(0,\pi]$ let $\theta\in[-\pi,\pi]$ with $|\theta|\ge\ep$ and $\lambda\in\sigma(T)$ be such that $p(\ep)=|e^{i\theta}-\lambda|\inv.$ Then $p(\ep)\inv\ge1-r$, where $r=|\lambda|$. In particular, $\lambda\ne0$ when $\ep$ is sufficiently small. In fact, since $p(\ep)\ge\ep\inv$ for $0<\ep\le\pi$ we see that $r\to1-$ and $\theta\to0$ as $\ep\to0+$. In particular, $1-r\ge\frac12\log (\frac1r)$ when $\ep$ is sufficiently small. Let $t\ge1$ and suppose for the moment that $\ep\in(0,\pi]$ is such that $p(t\inv\ep)=\delta bn$ for some $b\in(c/2,c]$ and some $n\ge1$. Then $p\inv(\delta bn)\inv\ge t\ep\inv$. Hence if $|1-\lambda|\ge\ep/2$ then $\lambda\ne1$ and \eqref{eq:est} yields
$$\frac{p(t\inv\ep)}{p(\ep)}\ge\frac{\delta b}{2}\log\left(\frac{t}{2C}\right)$$
provided $\ep$ is sufficiently small. On the other hand, if $|1-\lambda|<\ep/2$ then
$p(\ep)\inv\ge|1-e^{i\theta}|-|1-\lambda|\ge\ep/3$ if $\ep$ is sufficiently small, and hence $p(t\inv\ep)/p(\ep)\ge t/3.$ It follows that there exists $t>1$ such that $p(t\inv\ep)/p(\ep)\ge2\delta\inv$ for all sufficiently small values of $\ep\in(0,\pi] $ such that $p(t\inv\ep)=\delta bn$ for some $b\in(c/2,c]$ and some $n\ge1$. To handle the general case note that if $\ep\in(0,\pi]$ is sufficiently small then $c(n-1)<\delta\inv p(t\inv\ep)\le cn$ for some $n\ge2$. Let $b=\delta\inv n\inv p(t\inv\ep)$. Then $p(t\inv\ep)=\delta b n$ and 
$$\frac{c}{2}\le\frac{c(n-1)}{n}<b\le c,$$
and hence by the first part $p(t\inv\ep)/p(\ep)\ge2\delta\inv$. Define $M\colon[\pi\inv,\infty)\to(0,\infty)$ by $M(s)=m(s\inv)$, $s\ge\pi\inv$. Then
$$\liminf_{s\to\infty}\frac{M(t s)}{M(s)}=\liminf_{\ep\to0+}\frac{m(t\inv\ep)}{m(\ep)}\ge \delta\liminf_{\ep\to0+}\frac{p(t\inv\ep)}{p(\ep)}\ge2>1,$$
and it follows from \cite[Lem.~2.1]{RSS19} (see also \cite[Def.~2]{deHSta85}) that $M$ has positive increase, so $m$ has reciprocally positive inverse, as required. 
\end{proof}

\section{Quasi-multiplication operators}\label{sec:QMO}

We now consider the class of so-called \emph{quasi-multiplication operators}, that is to say bounded linear operators $T$ on a  Banach space such that
\begin{equation}\label{eq:qm}
\|f(T)\|=\sup_{\lambda\in\sigma(T)}|f(\lambda)|
\end{equation}
for all rational functions $f$ which have no poles in $\sigma(T)$. This class has previously been considered in \cite{BaChTo16, RSS19, Sei16}. Canonical examples of quasi-multiplication operators include multiplication operators on classical function spaces. Within the Hilbert space context the class contains all normal operators but, as we now show, it is possible for an an operator on a Hilbert space to be a quasi-multiplication operator without being normal.

\begin{ex}\label{ex:QMO}
Using the notation of Section~\ref{sec:Toep}, we see easily that $T_a$ is a non-normal quasi-multiplication operator on $H^2(\TT)$ for any density $a\in P(\ZZ_+)\setminus\{(1,0,0,\dots)\}$. Indeed, if $f$ is a rational function with no poles in $\sigma(T_a)=\phi_a(\overline{\DD})$ then $f(T_a)$ is the analytic Toeplitz operator with symbol $f\circ\phi_a\in H^\infty(\TT)$ and by the maximum modulus principle we have
$$\|f(T_a)\|=\|f\circ\phi_a\|_{L^\infty(\TT)}=\sup_{\lambda\in\sigma(T_a)}|f(\lambda)|,$$
as required.  More generally, \emph{any} analytic Toeplitz operator on $H^2(\TT)$ is a quasi-multipli\-cation operator.
\end{ex}

We conclude our paper with the following result for quasi-multiplication operators, which  gives improved upper and lower bounds in Theorem~\ref{thm:KT} involving no implicit constants. The result may be regarded as a discrete analogue of \cite[Th.~4.4]{RSS19}, but the proof is more involved and the lower bound is marginally less sharp. The latter appears to be unavoidable.

\begin{thm}\label{thm:quasi_mult}
Let $T$ be a quasi-multiplication operator such that $\{1\}\subseteq\sigma(T)\subseteq\DD\cup\{1\}$, and define the functions $m$ and $\mmax$ by
$$m(\ep)=\max_{\ep\le|\theta|\le\pi}\frac{1}{\dist(e^{i\theta},\sigma(T))},\quad 0<\ep\le\pi,$$
and 
$$\mmax(\ep)=\max_{\ep\le\theta\le\pi} m(\theta)\log\left(\frac{\theta}{\ep}\right),\quad 0<\ep\le\pi,$$
respectively. 
\begin{enumerate}
\item[\textup{(i)}] Suppose that $\delta\in(0,1]$ is such that
\begin{equation}\label{eq:ass_delta}
\min_{\ep\le|\theta|\le\pi}\dist(e^{i\theta},\sigma(T))\le\delta\ep
\end{equation}
for all sufficiently small $\ep\in(0,\pi]$. Then 
$$\|T^n(I-T)\|\le (1+\delta)\mmax\inv(n)$$
for all sufficiently large $n\ge1$. 
\item[\textup{(ii)}] Suppose that \eqref{eq:ass_delta} holds for some $\delta\in(0,1)$ and all sufficiently small $\ep\in(0,\pi]$, and let $\delta'\in(\delta,1)$ and  $c>1$. Then
\begin{equation}\label{eq:mmax_lb}
\|T^n(I-T)\|\ge (1-\delta')\mmax\inv(cn)
\end{equation}
for all sufficiently large $n\ge1$.
\end{enumerate}
\end{thm}

\begin{rem}
Note that condition \eqref{eq:ass_delta} certainly holds for $\delta=1$, as a simple consequence of the fact that $1\in\sigma(T)$, and hence we always have
$\|T^n(I-T)\|\le 2\mmax\inv(n)$ for all sufficiently large $n\ge1$. If 
$$\min_{\ep\le|\theta|\le\pi}\dist(e^{i\theta},\sigma(T))=o(\ep),\quad \ep\to0+,$$  
then \eqref{eq:ass_delta} holds for \emph{all} $\delta\in(0,1]$ and \eqref{eq:mmax_lb} holds for all $\delta'\in(0,1)$. The example of the identity operator shows that there need not be a non-trivial lower bound if \eqref{eq:ass_delta} does not hold for any $\delta\in(0,1)$. 
\end{rem}

\begin{proof}[Proof of Theorem~\ref{thm:quasi_mult}]
We begin by proving part (i). Since $T$ is a quasi-multiplication operator we have 
\begin{equation}\label{eq:qm_id}
\|T^n(I-T)\|=\sup_{\lambda\in\sigma(T)}|p_n(\lambda)|,\quad n\ge0,
\end{equation}
where $p_n(\lambda)=\lambda^n(1-\lambda)$. Let us define  regions $\Omega_\delta$ and $\Theta_\delta$ by
$$\begin{aligned}
\Omega_\delta&=\big\{re^{i\theta}:  -\pi\le\theta\le\pi,0\le r\le\max\{0,1-\delta|\theta|\}\big\},\\
\Theta_\delta&=\big\{re^{i\theta}:  -\pi\le\theta\le\pi,\max\{0,1-\delta|\theta|\}\le r\le1\big\},
\end{aligned}$$
so that the union of $\Omega_\delta$ and $\Theta_\delta$ is the closed unit disk. In particular, $\sigma(T)\subseteq\Omega_\delta\cup\Theta_\delta$. 
We note first that, by the maximum modulus principle, 
$$\sup_{\lambda\in\sigma(T)\cap\Omega_\delta}|p_n(\lambda)|\le\sup_{\lambda\in\partial\Omega_\delta}|p_n(\lambda)|.$$
Now for $\lambda\in\partial\Omega_\delta$  we may write $\lambda=re^{i\theta}$ with  $|\theta|\le\min\{\pi,\delta\inv\}$ and $r=1-\delta|\theta|$. Then 
\begin{equation}\label{eq:easy_est}
|1-\lambda|\le|e^{i\theta}-re^{i\theta}|+|1-e^{i\theta}|\le1-r+|\theta|=(1+\delta)|\theta|,
\end{equation}
 and hence
$$|p_n(\lambda)|\le(1+\delta)|\theta|\big(1-\delta|\theta|\big)^n,\quad n\ge0.$$
The map $t\mapsto t(1-t)^n$ attains its maximum on $[0,1]$ at $t=(n+1)\inv$, so
$$\sup_{\lambda\in\partial\Omega_\delta}|p_n(\lambda)|\le\frac{1+\delta\inv}{n+1}\left(1-\frac{1}{n+1}\right)^n\le\frac{1+\delta\inv}{en},\quad n\ge1.
$$
Let $\ep_n=\mmax\inv(n)$, $n\ge1$. Then $n=\mmax(\ep_n)\ge m(\ep_n e)$ provided that $\ep_n\le\pi/e$, and hence by  \eqref{eq:ass_delta} we have $n\ge\delta\inv\ep_n\inv e\inv$ if $n$ is sufficiently large. It follows that $e\inv n\inv\le\delta\mmax\inv(n)$ and consequently  
\begin{equation}\label{eq:Omega_est}\sup_{\lambda\in\sigma(T)\cap\Omega_\delta}|p_n(\lambda)|\le (1+\delta)\mmax\inv(n)
\end{equation}
for all sufficiently large $n\ge1$. Now suppose that $\lambda\in\sigma(T)\cap\Theta_\delta$ and write $\lambda=re^{i\theta}$ with $-\pi\le\theta\le\pi$ and $\max\{0,1-\delta|\theta|\}\le r\le1$. Then $1-r\le\delta|\theta|$ and hence as in \eqref{eq:easy_est} we have $|1-\lambda|\le (1+\delta)|\theta|$. Using the fact that $m(|\theta|)\ge(1-r)\inv$ we obtain
$$|p_n(\lambda)|\le |1-\lambda|\left(1-\frac{1}{m(|\theta|)}\right)^n\le (1+\delta)|\theta|\exp\left(-\frac{n}{m(|\theta|)}\right),\quad n\ge1.$$
If $|\theta|\le\ep_n=\mmax\inv(n)$ then  $|p_n(re^{i\theta})|\le(1+\delta)\mmax\inv(n)$. Suppose that $\ep_n\le|\theta|\le\pi$. Then 
$$n=\mmax(\ep_n)\ge m(|\theta|)\log\left(\frac{|\theta|}{\ep_n}\right),$$
and hence 
$$|\theta|\exp\left(-\frac{n}{m(|\theta|)}\right)\le\mmax\inv(n),\quad n\ge1.$$
Thus
$$\sup_{\lambda\in\sigma(T)\cap\Theta_\delta}|p_n(\lambda)|\le(1+\delta)\mmax\inv(n),\quad n\ge1,$$
and combining this with \eqref{eq:Omega_est} in \eqref{eq:qm_id} yields the result.

In order to prove part~(ii), suppose that \eqref{eq:ass_delta} holds for some $\delta\in(0,1)$ and all sufficiently small $\ep\in(0,\pi]$, and let $\delta'\in(\delta,1)$ and $c>1$. Define $\omega\colon\ZZ_+\to(0,\infty)$ by $\omega(n) = \|T^n(I-T)\|,$ $n\in \ZZ_+$, noting that $\omega(n)\to0$ as $n\to\infty$ by Theorem~\ref{thm:KT}, and define the maps $\omega_*$ and $\omega^*$ by
$$\begin{aligned}
\omega_*(s)&=\min\{n\in\ZZ_+:\omega(n)\le s\},\quad &&s>0,\\
\omega^*(s)&=\max\{n\in\ZZ_+:\omega(n)\ge s\},\quad &&0<s\le \omega(0).
\end{aligned}$$
We have $\omega_*(s),\omega^*(s)\to\infty$ as $s\to0+$, and moreover $\omega(\omega_*(s))\le s$ for all $s>0$ while $\omega(\omega^*(s))\ge s$ for $0<s\le\omega(0)$. Furthermore, $\omega(\omega^*(s)+1)<s$  and hence $\omega_*(s)\le\omega^*(s)+1$ for $0<s\le\omega(0)$. By quasi-multiplicativity of $T$ we have $\omega(n) \geq |p_n(\lambda)|$ for all  $\lambda \in \sigma(T)$ and $n\geq 0,$ and therefore
\begin{equation}\label{eq:log_bd}
n\log\left(\frac{1}{|\lambda|}\right)\ge\log \left(\frac{|1-\lambda|}{\omega(n)}\right), \quad \lambda \in \sigma(T)\setminus\{0,1\},\ n\geq0.
\end{equation}
Let $\ep\in(0,\pi]$ and let $t\ge1$ be such that $t\ep\le\pi$. We may find $\theta\in[-\pi,\pi]$ with $|\theta|\ge t\ep$ and $\lambda\in\sigma(T)$ such that $m(t\ep)=|e^{i\theta}-\lambda|\inv$. As in the proof of Proposition~\ref{prp:pos_inc_nec}, letting $r=|\lambda|$ we have $r\to1-$ and $\theta\to0$ as $t\ep\to0$. Hence  given  $c'\in(1,c)$ we may find $\theta_0\in(0,\pi]$ such that $1-r\ge \frac{1}{c'}\log (\frac1r)$ and $|1-e^{i\theta}|\ge (1+\delta-\delta')|\theta|$ whenever $t\ep\in(0,\theta_0)$. By making $\theta_0$ smaller if necessary we may further assume that $\lambda\ne0$ and, by virtue of \eqref{eq:ass_delta}, that  $|e^{i\theta}-\lambda|=m(t\ep)\inv\le \delta t\ep$ for $t\ep\in(0,\theta_0)$. The latter implies that $$|1-\lambda|\ge|1-e^{i\theta}|-|e^{i\theta}-\lambda|\ge(1-\delta')t\ep,$$
and in particular $\lambda\ne1$, whenever $ t\ep\in(0,\theta_0)$. Hence \eqref{eq:log_bd} yields
\begin{equation}\label{eq:minv_est}
m(t\ep)\inv=|e^{i\theta}-\lambda|\ge 1-r\ge \frac{1}{c'}\log\left(\frac1r\right)\ge\frac{1}{c'n}\log\left(\frac{(1-\delta')t\ep}{\omega(n)}\right)
\end{equation}
for all $n\ge1$ provided that $t\ep\in(0,\theta_0)$. Since $t\ge1$  we may assume, once again by making $\theta_0$ smaller if necessary, that $\omega_*((1-\delta')\ep)\ge1$ for $t\ep\in(0,\theta_0)$, and hence setting $n=\omega_*((1-\delta')\ep)$ in \eqref{eq:minv_est}  yields
\begin{equation}\label{eq:omega_est}
\omega_*\big((1-\delta')\ep\big)\ge \frac{1}{c'}m(t\ep)\log(t),\quad t\ep\in(0,\theta_0).
\end{equation}
It is straightforward to see that 
$$\mmax(\ep)=\max_{\ep\le\vartheta\le\theta_0}m(\vartheta)\log\left(\frac{\vartheta}{\ep}\right)$$
provided $\ep\in(0,\theta_0)$ is sufficiently small. We may combine this in \eqref{eq:omega_est} with our earlier observations about $\omega_*$ and $\omega^*$ to deduce that $\omega^*((1-\delta')\ep)\ge c\inv\mmax(\ep)$ for all sufficiently small $\ep\in(0,\pi]$. Setting $\ep=\mmax\inv(cn)$ for sufficiently large $n\ge1$ and applying the decreasing map $\omega$ to both sides of the resulting estimate, we obtain 
$$\omega(n)\ge\omega\big(\omega^*\big((1-\delta')\mmax\inv(cn)\big)\big)\ge(1-\delta')\mmax\inv(cn),$$
which completes the proof.
\end{proof}

We conclude by comparing the bounds obtained in Theorem~\ref{eq:mmax_lb} with the results discussed elsewhere in the paper. Let us suppose, to this end, that the function $m$ considered in Theorem~\ref{eq:mmax_lb} satisfies $m(\ep)\ge c_\alpha\ep^{-\alpha}$ for some constants $\alpha\ge1$, $c_\alpha>0$ and all sufficiently small $\ep\in(0,\pi]$, noting that this always holds for $\alpha=1$ and $c_1=1$. Then for any $c\in(0,1+\alpha)$ we have $\mmax(\ep)\le c\inv \mlog(\ep)$ for all sufficiently small $\ep\in(0,\pi]$, and hence $\mmax\inv(n)\le \mlog\inv(cn)$ for all sufficiently large $n\ge1$. In particular, for quasi-multi\-pli\-cation operators Theorem~\ref{thm:quasi_mult} produces a sharper version of \eqref{eq:mlog}, involving no implicit `big O' constant and a larger range of admissible values of $c>0$. Similarly, if we assume that $m$ has reciprocally positive increase then one may show that $\mmax(\ep)\le (c \alpha e)\inv m(\ep)$ for all sufficiently small $\ep\in(0,\pi]$, where $c,\alpha>0$ are as in \eqref{eq:pos_inc}. Using \cite[Prop.~2.2]{RSS19} it follows that $\mmax\inv(n)=O(m\inv(n))$ as $n\to\infty$, so we recover Theorem~\ref{thm:pos_inc} in our special setting of quasi-multi\-pli\-cation operators. We note in passing that, by Proposition~\ref{prp:pos_inc_nec} and Theorem~\ref{thm:quasi_mult}, an estimate of the form $\mmax\inv(n)=O(m\inv(cn))$ as $n\to\infty$ with $c>0$ can hold \emph{only} if $m$ has positive increase.
Finally, in order to compare the lower bound in \eqref{eq:mmax_lb} with Proposition~\ref{prp:lb} and  \cite[Cor.~2.6]{Sei16}, observe that for any $c,c'>0$ we have $\mmax\inv(cn)\ge e^{-c/c'}m\inv(c'n)$ when $n\ge1$ is sufficiently large.

\medskip

\noindent \textbf{Acknowledgement.} The authors wish to thank Prof.\ Charles Batty for helpful discussions on the topic of this paper.

\end{document}